\documentclass[12pt]{article}

\usepackage{latexsym}
\usepackage{amsxtra}
\usepackage{amssymb}
\usepackage{bm}
\usepackage{amsthm}
\usepackage{mathtools}
\usepackage{xspace}
\usepackage{graphicx}
\usepackage[margin=1in]{geometry}
\usepackage{xfrac}

\input xy
\xyoption{all} \CompileMatrices

\begin{document}

\def\HH{{\mathcal{H}}}
\def\orb{{\operatorname{orb}}}
\def\diam{{\operatorname{diam}}}
\def\II{{\mathfrak{I}}}
\def\PO{{\operatorname{PO}}}
\def\Cl{{\operatorname{Cl}}}
\def\Max{{\operatorname{-Max}}}
\def\XX{{\mathfrak X}}
\def\YY{{\bf{Y}}}
\def\BBB{{\mathcal B}}
\def\inv{{\operatorname{inv}}}
\def\emph{\it}
\def\Int{{\operatorname{Int}}}
\def\Spec{\operatorname{Spec}}
\def\Bin{{\operatorname{B}}}
\def\n{\operatorname{b}}
\def\N{{\operatorname{GB}}}
\def\BC{{\operatorname{BC}}}
\def\dlog{\frac{d \log}{dT}}
\def\Sym{\operatorname{Sym}}
\def\Nr{\operatorname{Nr}}
\def\lbrack{{\{}}
\def\rbrack{{\}}}
\def\burnside{\operatorname{B}}
\def\Sym{\operatorname{Sym}}
\def\Hom{\operatorname{Hom}}
\def\Inj{\operatorname{Inj}}
\def\Aut{{\operatorname{Aut}}}
\def\Mor{{\operatorname{Mor}}}
\def\Map{{\operatorname{Map}}}
\def\CMap{{\operatorname{CMap}}}
\def\GMaps{G{\operatorname{-Maps}}}
\def\Fix{{\operatorname{Fix}}}
\def\res{{\operatorname{res}}}
\def\ind{{\operatorname{ind}}}
\def\inc{{\operatorname{inc}}}
\def\coind{{\operatorname{cnd}}}
\def\Equiv{{\mathcal{E}}}
\def\W{\operatorname{W}}
\def\F{\operatorname{F}}
\def\witt{\operatorname{gh}}
\def\ngh{\operatorname{ngh}}
\def\Fm{{\operatorname{Fm}}}
\def\bij{{\iota}}
\def\mk{{\operatorname{mk}}}
\def\km{{\operatorname{mk}}}
\def\VV{{\bf{V}}}
\def\ff{{\bf{f}}}
\def\ZZ{{\mathbb Z}}
\def\Zhat{{\widehat{\mathbb Z}}}
\def\CC{{\mathbb C}}
\def\PP{{\mathbf p}}
\def\L{{\mathbf L}}
\def\DD{{\mathbb D}}
\def\EE{{\mathbb E}}
\def\MM{{\mathbb M}}
\def\JJ{{\mathbb J}}
\def\NN{{\mathbb N}}
\def\RR{{\mathbb R}}
\def\QQ{{\mathbb Q}}
\def\FF{{\mathbb F}}
\def\mm{{\mathfrak m}}
\def\nn{{\mathfrak n}}
\def\jj{{\mathfrak j}}
\def\aaa{{{{\mathfrak a}}}}
\def\bbb{{{{\mathfrak b}}}}
\def\ppp{{{{\mathfrak p}}}}
\def\qqq{{{{\mathfrak q}}}}
\def\PPP{{{{\mathfrak P}}}}
\def\BB{{\mathfrak B}}
\def\jj{{\mathfrak J}}
\def\LL{{\mathfrak L}}
\def\qq{{\mathfrak Q}}
\def\rr{{\mathfrak R}}
\def\cc{{\mathfrak S}}
\def\TT{{\mathcal{T}}}
\def\SS{{\mathcal S}}
\def\UU{{\mathcal U}}
\def\AA{{\mathcal A}}
\def\BB{{\mathcal B}}
\def\Primes{{\mathcal P}}
\def\genS{{\langle S \rangle}}
\def\genT{{\langle T \rangle}}
\def\bT{\mathsf{T}}
\def\bD{\mathsf{D}}
\def\bC{\mathsf{C}}
\def\VV{{\bf V}}
\def\ff{{\bf f}}
\def\uu{{\bf u}}
\def\aa{{\bf{a}}}
\def\bb{{\bf{b}}}
\def\zero{{\bf 0}}
\def\rad{\operatorname{rad}}
\def\End{\operatorname{End}}
\def\id{\operatorname{id}}
\def\mod{\operatorname{mod}}
\def\im{\operatorname{im}\,}
\def\ker{\operatorname{ker}}
\def\coker{\operatorname{coker}}
\def\ord{\operatorname{ord}}
\def\li{\operatorname{li}}
\def\Ei{\operatorname{Ei}}
\def\Ein{\operatorname{Ein}}
\def\Ri{\operatorname{Ri}}
\def\Rie{\operatorname{Rie}}
\def\degl{\operatorname{deglog}}

\newtheorem{theorem}{Theorem}[section]
\newtheorem*{theorem*}{Theorem}
\newtheorem{proposition}[theorem]{Proposition}
\newtheorem{corollary}[theorem]{Corollary}
\newtheorem{lemma}[theorem]{Lemma}
\newtheorem{unnumblemma}{Lemma}[section]

\theoremstyle{definition}
\newtheorem{example}[theorem]{Example}
\newtheorem{definition}[theorem]{Definition}
\newtheorem*{definition*}{Definition}
\newtheorem{remark}[theorem]{Remark}
\newtheorem{problem}[theorem]{Problem}
\newtheorem{conjecture}[theorem]{Conjecture}

 \newenvironment{map}[1]
   {$$#1:\begin{array}{rcl}}
   {\end{array}$$
   \\[-0.5\baselineskip]
 }

 \newenvironment{map*}
   {\[\begin{array}{rcl}}
   {\end{array}\]
   \\[-0.5\baselineskip]
 }

 \newenvironment{nmap*}
   {\begin{eqnarray}\begin{array}{rcl}}
   {\end{array}\end{eqnarray}
   \\[-0.5\baselineskip]
 }

 \newenvironment{nmap}[1]
   {\begin{eqnarray}#1:\begin{array}{rcl}}
   {\end{array}\end{eqnarray}
   \\[-0.5\baselineskip]
 }

\newcommand{\eq}{eq.\@\xspace}
\newcommand{\eqs}{eqs.\@\xspace}
\newcommand{\diagram}{diag.\@\xspace}

\numberwithin{equation}{section}


\title{Asymptotic expansions of weighted prime power counting functions}

\author{Jesse Elliott \\  
California State University, Channel Islands \\
{\tt jesse.elliott@csuci.edu}}

\maketitle

\begin{abstract}
We prove several asymptotic continued fraction expansions of $\pi(x)$, $\Pi(x)$, $\li(x)$, $\Ri(x)$, and related functions, where $\pi(x)$ is the prime counting function, $\Pi(x) =  \sum_{k = 1}^\infty \frac{1}{k}\pi(\sqrt[k]{x})$ is the Riemann prime counting function, and $\Ri(x) = \sum_{k=1}^\infty \frac{ \mu(k)}{k} \li(\sqrt[k]{x})$ is Riemann's approximation to the prime counting function.   We also determine asymptotic continued fraction expansions of the function $\sum_{p \leq x} p^s$ for all $s \in \CC$ with $\operatorname{Re}(s) > -1$, and of the functions $\sum_{a^x < p \leq a^{x+1}} \frac{1}{p}$ and $\log \prod_{a^x < p \leq a^{x+1}} (1 -1/p)^{-1}$ for all real numbers $a > 1$.  We also determine the first few  terms of an asymptotic continued fraction expansion of the function $\pi(ax)-\pi(bx)$ for  $a > b > 0$.  As a corollary of these results, we determine the best rational approximations of the ``linearized'' verions of these various functions.  \\

\noindent {\bf Keywords:}  prime counting function, asymptotic expansion, continued fraction. \\

\noindent {\bf MSC:}   11N05, 30B70,  44A15
\end{abstract}

\bigskip

\tableofcontents

\section{Introduction}

\subsection{Summary}

This paper concerns the asymptotic behavior of the function $\pi: \RR_{>0} \longrightarrow \RR$ that for any $x > 0$ counts the number of primes less than or equal to $x$: $$\pi(x) =   \# \{p \leq x: p \mbox{ is prime}\}, \quad x > 0.$$  The function $\pi(x)$ is known as the {\bf prime counting function}.  The celebrated {\bf prime number theorem}, proved independently by de la Vall\'ee Poussin \cite{val1} and Hadamard \cite{had}  in 1896,  states that
\begin{align*}
\pi(x) \sim \frac{x}{\log x} \ (x \to \infty),
\end{align*}
where $\log x$ is the natural logarithm.  It is well known that this is just the first term of a (divergent) asymptotic expansion of $\pi(x)$, namely, 
\begin{align*}
\frac{\pi(x)}{x} \sim \sum_{n = 1}^\infty \frac{n!}{(\log x)^n} \ (x \to \infty).
\end{align*}
As shown in \cite[Theorem 1.1]{ell}, this can be reinterpreted as the (divergent) asymptotic continued fraction expansions
$$\frac{\pi(x)}{x}  \sim \,  \cfrac{\frac{1}{\log x}}{1 \,-} \ \cfrac{\frac{1}{\log x}}{1 \,-}\  \cfrac{\frac{1}{\log x}}{1 \,-}\  \cfrac{\frac{2}{\log x}}{1 \,-}\  \cfrac{\frac{2}{\log x}}{1 \,-}\  \cfrac{\frac{3}{\log x}}{1 \,-} \  \cfrac{\frac{3}{\log x}}{1 \,-}\  \cfrac{\frac{4}{\log x}}{1 \,-}  \ \cfrac{\frac{4}{\log x}}{1 \,-}  \ \cdots \ (x \to \infty)$$
and
$$\frac{\pi(x)}{x}  \,  \sim \, \cfrac{1}{\log x -1\,-} \  \cfrac{1}{\log x  - 3 \,-}\  \cfrac{4}{\log x-5\,-}\  \cfrac{9}{\log x - 7 \,-} \ \cfrac{16}{\log x-9\,-}\   \cdots \  (x \to \infty).$$
In this paper, we prove similar asymptotic continued fraction expansions of various weighted prime power counting functions and their smooth approximations.

Specifically,  we prove several asymptotic continued fraction expansions of $\pi(x)$, $\Pi(x)$, $\li(x)$, $\Ri(x)$, and related functions, where $\pi(x)$ is the prime counting function, $\Pi(x) =  \sum_{k = 1}^\infty \frac{1}{k}\pi(\sqrt[k]{x})$ is the Riemann prime counting function, and $\Ri(x) = \sum_{k=1}^\infty \frac{ \mu(k)}{k} \li(\sqrt[k]{x})$ is Riemann's approximation to the prime counting function.   We also determine asymptotic continued fraction expansions of the function $\sum_{p \leq x} p^s$ for all $s \in \CC$ with $\operatorname{Re}(s) > -1$, and of the functions $\sum_{a^x < p \leq a^{x+1}} \frac{1}{p}$ and $\log \prod_{a^x < p \leq a^{x+1}} (1 -1/p)^{-1}$ for all real numbers $a > 1$.  We also determine the first few  terms of an asymptotic continued fraction expansion of the function $\pi(ax)-\pi(bx)$ for  $a > b > 0$.  As a corollary of these results, we determine the best rational approximations of the ``linearized'' versions of these various functions.

This paper  is a sequel to the paper  \cite{ell}, and the definitions and results therein will be assumed here.   Thus, for example, we require the notion of an  {\it asymptotic expansion}, and that of an {\it asymptotic continued fraction expansion}, over some unbounded subset $\XX$ of $\CC$.   We also require the notions of a {\it Jacobi continued fraction},  a {\it Stieltjes continued fraction}, and a {\it best rational function approximation} of a function.

The paper \cite{ell} focuses on {\it divergent} asymptotic continued fraction expansions.  This paper deals also with {\it convergent} asymptotic continued fraction expansions.  In Section 1.2, we make a few general observations about such expansions.  In Section 2, we prove various asymptotic expansions of weighted prime power counting functions {\it relative to each other}.  Some of these asymptotic expansions are easily verified (e.g., $\Pi(x) \sim \sum_{n = 1}^\infty \frac{1}{n}\pi(\sqrt[n]{x})  \ (x \to \infty)$), but others, especially Propositions \ref{RAprop} and \ref{RApropa}, are undoubtedly worth making explicit.  Finally, in Section 3, we apply the results of Section 2 and of the paper  \cite{ell} to prove several asymptotic continued fraction expansions of various weighted prime power counting functions and their smooth approximations.

\subsection{Asymptotic continued fraction expansions}

The following result is an immediate corollary of \cite[Theorems 2.4 and 2.9]{ell}.

\begin{proposition}\label{aprop}
Let $f(z)$ be a complex-valued function defined on some unbounded subset $\XX$ of $\CC$, and let $\mu$ be a measure on $\RR$ with infinite support and finite moments $\mu_k = m_k(\mu) = \int_{-\infty}^\infty t^k d\mu \in \RR$.  Then the following conditions are equivalent.
\begin{enumerate}
\item One has the asymptotic expansion $$f(z) \sim \sum_{k = 0}^\infty \frac{\mu_k}{z^{k+1}} \ (z \to \infty)_\XX$$ of $f(z)$ over $\XX$.
\item $f(z)$ has an  asymptotic Jacobi continued fraction expansion
$$f(z) \, \sim \,\cfrac{a_1}{z+b_1 \,-} \  \cfrac{a_2}{z+b_2 \,-} \  \cfrac{a_3}{z+b_3 \,-} \  \cdots  \ (z \to \infty)_\XX$$
such that the $n$th approximant $w_n(z)$ of the continued fraction for all $n \geq 1$  has the asymptotic expansion
$$w_n(z) \sim  \sum_{k = 0}^{2n-1} \frac{\mu_k}{z^{k+1}}  \ (z \to \infty)$$ 
of order $2n$ at $z = \infty$, where $a_n, b_n \in \RR$ and $a_n > 0$ for all $n$.
\end{enumerate} 
If the conditions above hold, then $f(z)$ and the sequences $\{a_n\}$ and $\{b_n\}$ satisfy the equivalent conditions (2)(a)--(e) of \cite[Theorem 2.4]{ell}, and so, for example, the $w_n(z)$ are precisely the best rational approximations of $f(z)$ over $\XX$.
\end{proposition}

The hypotheses on $f(z)$ of the proposition can be achieved, at least over any subset $\XX$ of $\CC_{\delta, \varepsilon} =  \{z \in \CC : \delta \leq |\operatorname{Arg}(z)| \leq \pi- \varepsilon\}$, for any fixed $\delta, \varepsilon > 0$, if one has
$$f(z) = {\mathcal S}_\mu(z) + O\left(\frac{1}{z^k} \right) \ (z \to \infty)_\XX$$
for all integers $k$,  where ${\mathcal S}_\mu(z) = \int_{-\infty}^\infty \frac{d \mu(t)}{z-t}$ denotes the {\bf Stieltjes transform} of $\mu$.  However, in cases relevant to various prime counting functions one seeks an asymptotic expansion of some function $f(x)$ over $\RR_{> 0}$, not over $\CC_{\delta, \varepsilon}$.  In these cases, one would have to verify the asymptotic expansion $f(x) \sim \sum_{k = 0}^\infty \frac{\mu_k}{x^{k+1}} \ (x \to \infty)$ over $\RR_{>0}$ through some other means.  As discussed in \cite{ell}, this  exact situation occurs, for example, with the function $f(x) = \frac{\pi(e^x)}{e^x}$, since the asymptotic expansion $\frac{\pi(e^x)}{e^x} \sim \sum_{k = 0}^\infty \frac{\mu_k(\gamma_0)}{x^{k+1}} \ (x \to \infty)$ follows from the prime number theorem with error term, where $\gamma_0$ is the exponential distribution with weight parameter $1$ supported on $[0,\infty)$.

If, however, $\mu$ a finite measure on $\RR$ with infinite and compact support, then $\mu$ has finite moments, and the asymptotic continued  fraction expansion of ${\mathcal S}_\mu(z)$ in  holds over $\CC$, not just over $\CC_{\delta, \varepsilon} = \{z \in \CC : \delta \leq |\operatorname{Arg}(z)| \leq \pi- \varepsilon\}$, that is, one has
$${\mathcal S}_\mu(z) \, \sim \, \cfrac{a_1}{z+b_1 \,-} \  \cfrac{a_2}{z+b_2 \,-} \  \cfrac{a_3}{z+b_3 \,-} \  \cdots  \ (z \to \infty).$$
In this case, a function $f(z)$ has the asymptotic Jacobi continued fraction expansion 
$$f(z) \, \sim \, \cfrac{a_1}{z+b_1 \,-} \  \cfrac{a_2}{z+b_2 \,-} \  \cfrac{a_3}{z+b_3 \,-} \  \cdots  \ (z \to \infty)_\XX$$ over some unbounded subset $\XX$ of $\CC$  if and only if
$$f(z) = {\mathcal S}_\mu(z) + O\left(\frac{1}{z^k} \right) \ (z \to \infty)_\XX$$
for all integers $k$.  Also in this case, ${\mathcal S}_\mu(z)$ is analytic at $\infty$ with Laurent expansion
$${\mathcal S}_\mu(z) = \sum_{k = 0}^\infty \frac{\mu_k}{z^{k+1}}, \quad |z| \gg 0$$
 and  Stieltjes continued fraction expansion
$${\mathcal S}_\mu(z) = \cfrac{a_1}{z+b_1 \,-} \  \cfrac{a_2}{z+b_2 \,-} \  \cfrac{a_3}{z+b_3 \,-} \ \cdots, \quad z \in \CC\backslash \RR \text{ or } |z| \gg 0.$$

A proposition analogous to Proposition \ref{aprop}, along with similar comments, hold for  Stieltjes continued fractions, as a consequence of \cite[Theorem 2.6 and 2.8]{ell}.

\begin{example}
For a simple example that will be relevant in Section 3.2, consider the uniform distribution $\mu$ on $[-1,0]$.  This measure has Stieltjes transform 
\begin{align}\label{logas1}
\log\left(1+\frac{1}{z}\right) = \cfrac{\frac{1}{z}}{1 \,+}\  \cfrac{\frac{1}{z}}{2 \,+}\  \cfrac{\frac{1}{z}}{3 \,+}\   \cfrac{\frac{4}{z}}{4 \,+}\  \cfrac{\frac{4}{z}}{5 \,+} \  \cfrac{\frac{9}{z}}{6 \,+} \  \cfrac{\frac{9}{z}}{7 \,+} \  \cfrac{\frac{16}{z}}{8 \,+} \  \cfrac{\frac{16}{z}}{9 \,+} \  \cdots, \quad z \in \CC\backslash [-1,0],
\end{align}
and moments $$m_n(\mu) = \int_{-1}^0 t^n \, dt = \frac{(-1)^n}{n+1},$$
which yields the asymptotic expansion
$$\log\left(1+\frac{1}{z}\right)  \sim  \sum_{n = 1}^\infty \frac{(-1)^{n-1}}{nz^{n}} \ (z \to \infty),$$
which of course is also valid as an exact Laurent expansion for $|z| > 1$ (where $\log$ is the principal branch of the logarithm).  Consequently, one also has the asymptotic expansion
\begin{align}\label{logas}
\log\left(1+\frac{1}{z}\right) \sim \cfrac{\frac{1}{z}}{1 \,+}\  \cfrac{\frac{1}{z}}{2 \,+}\  \cfrac{\frac{1}{z}}{3 \,+}\   \cfrac{\frac{4}{z}}{4 \,+}\  \cfrac{\frac{4}{z}}{5 \,+} \  \cfrac{\frac{9}{z}}{6 \,+} \  \cfrac{\frac{9}{z}}{7 \,+} \  \cfrac{\frac{16}{z}}{8 \,+} \ \cdots \  (z \to \infty).
\end{align}
The expansion  (\ref{logas1}) is well-known and is re-expressed in the  form
\begin{align*}
\log (1+z) = \cfrac{z}{1 \,+}\  \cfrac{z}{2 \,+}\  \cfrac{z}{3 \,+}\   \cfrac{4z}{4 \,+}\  \cfrac{4z}{5 \,+} \  \cfrac{9z}{6 \,+} \  \cfrac{9z}{7 \,+} \  \cfrac{16z}{8 \,+} \  \cfrac{16z}{9 \,+}\ \cdots, \quad z \in \CC\backslash (-\infty,-1]
\end{align*}
via the transformation $z \longmapsto 1/z$.
\end{example}

Further examples, as they relate to the prime counting function, are provided in Section 3.

\section{Relative asymptotic expansions}

\subsection{The Riemann prime counting function}

The {\bf Riemann prime counting function} is given by
$$\Pi(x) = \sum_{n = 1}^\infty \frac{1}{n}\pi(\sqrt[n]{x}) = \sum_{n = 1}^\infty \sum_{p^n \leq x} \frac{1}{n}, \quad x >0.$$
It is a weighted prime power counting function, where each power $p^n > 1$ of a prime $p$ is weighted by $\frac{1}{n}$.   Since $\pi(x) = \Pi(x) = 0$ if $x < 2$, and $\sqrt[n]{x} < 2$ if $n > \log_2 x$, one has
\begin{align}\label{Pi2}
\Pi(x) = \sum_{n \leq \log_2 x} \frac{1}{n}\pi(\sqrt[n]{x}), \quad x > 0.
\end{align}

\begin{proposition}\label{RAprop2}
One has the asymptotic expansion
\begin{align*}
\Pi(x) \sim \sum_{n = 1}^\infty \frac{1}{n}\pi(\sqrt[n]{x})  \ (x \to \infty).
\end{align*}
\end{proposition}

\begin{proof}  By (\ref{Pi2}), for any positive integer $N$, one has $$\frac{1}{N}\pi(x^{1/N}) \leq  \Pi(x)-\sum_{k = 1}^{N-1} \frac{1}{n}\pi(x^{1/n}) \leq \frac{1}{N}\pi(x^{1/N}) + \frac{1}{N+1}(\log_2 x) \pi(x^{1/(N+1)})$$
for all $x > 2^N$, and therefore
$$1 \leq \frac{ \Pi(x)-\sum_{n = 1}^{N-1} \frac{1}{n}\pi(x^{1/n})}{\frac{1}{N}\pi(x^{1/N})} \leq 1  + \frac{\frac{1}{N+1}(\log_2 x )\pi(x^{1/(N+1)})}{\frac{1}{N}\pi(x^{1/N})} \to 1$$
as $x \to \infty$.  It follows that
$$\lim_{x \to \infty} \frac{ \Pi(x)-\sum_{n = 1}^{N-1} \frac{1}{n}\pi(x^{1/n})}{\frac{1}{N}\pi(x^{1/N})} = 1.$$
The proposition follows.
\end{proof}

\begin{corollary}
One has
\begin{align*} \Pi(x)-\pi(x)  =  \sum_{n = 2}^\infty \sum_{p^n \leq x} \frac{1}{n} \sim \frac{1}{2}\pi(\sqrt{x}) \sim \frac{\sqrt{x}}{\log x} \ (x \to \infty)
\end{align*}
and
\begin{align*} \Pi(x)-\pi(x) -\frac{1}{2}\pi(\sqrt{x}) \sim \frac{1}{3}\pi(\sqrt[3]{x}) \sim \frac{\sqrt[3]{x}}{\log x} \ (x \to \infty).
\end{align*}
\end{corollary}

As is well known, by M\"obius inversion one has
$$\pi(x) = \sum_{n \leq \log_2 x} \frac{ \mu(n)}{n} \Pi(\sqrt[n]{x}) = \sum_{n=1}^\infty \frac{ \mu(n)}{n} \Pi(\sqrt[n]{x}), \quad x > 0.$$
A proof similar to that of Proposition \ref{RAprop2} yields the following.

\begin{proposition}\label{RAprop3}  One has the asymptotic expansion
\begin{align*}
\pi(x) \sim \sum_{n = 1}^\infty \frac{\mu(n)}{n}\Pi(\sqrt[n]{x})  \ (x \to \infty).
\end{align*}
\end{proposition}

\subsection{Riemann's approximation to the prime counting function}

The logarithmic integral function $\li(x) = \int_0^x \frac{dt}{\log t}$ can be extended to a complex function by setting
$$\li(z) = \Ei(\log z),$$
where 
$$\Ei(z) = \gamma+ \log z - \Ein(-z) = \gamma + \log z + \sum_{k =1}^\infty \frac{z^k}{k\cdot k!}$$
and where
$$ \Ein(z) =\int _{0}^{z}(1-e^{-t}){\frac {dt}{t}}=\sum_{k=1}^{\infty }{\frac {(-1)^{k+1}z^{k}}{k \cdot k!}}$$
is entire.   Let
$$\pi_0(x) = \lim_{\epsilon \to 0} \frac{\pi(x+\epsilon)+ \pi(x-\epsilon)}{2}$$
and
$$\Pi_0(x) = \lim_{\epsilon \to 0} \frac{\Pi(x+\epsilon)+\Pi(x-\epsilon)}{2} = \sum_{n = 1}^\infty \frac{1}{n}\pi_0(x^{1/n}).$$
By M\"obius inversion one has
$$\pi(x)=\sum_{n=1}^\infty \frac{\mu(n)}{n}\Pi(x^{1/n}),$$
and likewise for $\pi_0(x)$.  {\bf  Riemann's explicit formula} for $\Pi_0$ states that
$$\Pi_0(x) = \li(x) - \sum_\rho \li(x^\rho) - \log 2, \quad x > 1,$$
where the sum runs over all of the zeros $\rho$ of the Riemann zeta function $\zeta(s)$ (the nontrivial zeros taken in conjugate pairs in order of increasing imaginary part and repeated to multiplicity).   {\bf  Riemann's explicit formula} for $\pi_0$ states that
$$\pi_0(x) = \Ri(x) - \sum_\rho \Ri(x^\rho), \quad x > 1,$$
where $\Ri(x)$ is Riemann's function
$$\Ri(x)=\sum_{n=1}^\infty \frac{ \mu(n)}{n} \li(x^{1/n}), \quad x > 0.$$  
It follows that the function $\li(x)$ is properly considered an approximation for $\Pi(x)$, while Riemann's function $\Ri(x)$ is the analogous approximation for $\pi(x)$.  





It is well known that $\li(x)$ has the series representation
\begin{align}\label{liseries}
\li(x) = \gamma + \log \log x + \sum_{k = 1}^\infty \frac{(\log x)^k}{k \cdot k!}, \  \ x> 1,
\end{align}
Similarly,  $\Ri(x)$ has the series representation
$$\Ri(x) = 1+ \sum_{k = 1}^\infty \frac{(\log x)^k}{k \cdot k!\, \zeta(k+1)},  \quad x > 1,$$
which is the well-known {\bf Gram series} representation of $\Ri(x)$.  


Let \begin{align*}
R(x) = \sum_{n \leq \log x} \frac{\mu(n)}{n} \li (x^{1/n}), \quad x > 1,
\end{align*}
so that
\begin{align}\label{RE}
R(e^x) = \sum_{n \leq x} \frac{\mu(n)}{n} \li (e^{x/n}), \quad x > 0.
\end{align}

\begin{lemma}[\cite{GH}]
One has
$$\Ri(x) = R(x) + O( (\log \log x)^2) \ (x \to \infty), \quad x > e.$$
\end{lemma}

\begin{proof}
The series representation (\ref{liseries}) for $\li(x)$
implies that
$$\li(t) =  \gamma + \log \log t  + O( \log t), \quad 1 < t < e,$$
hence also
$$\li(x^{1/n}) =  \gamma + \log \log x  - \log n + O\left(\frac{ \log x}{n}\right) \ (x \to \infty), \quad n > \log x,$$
where the implicit constant does not depend on $n$.  Therefore, using also the facts that  $\sum_{n = N}\frac{1}{n^2} \sim \frac{1}{N} \ (N \to \infty)$,  $\sum_{n = 1}^N \frac{1}{n} \sim \log N  \ (N \to \infty)$, $\sum_{n = 1}^\infty \frac{\mu(n)}{n}  = 0$, and $\sum_{n = 1}^\infty \frac{\mu(n)\log n}{n} = -1$, for  $x > e$ we have
\begin{align*}
\sum_{n > \log x} \frac{\mu(n)}{n} \li (x^{1/n}) &  = \sum_{n > \log x} \frac{\mu(n)}{n} (\gamma + \log \log x - \log n) +  O \left(  \sum_{n > \log x} \frac{ \log x}{n^2} \right) \\
 &  = \sum_{n > \log x} \frac{\mu(n)}{n} (\gamma + \log \log x - \log n) + O(1) \\
& = \sum_{n = 1}^\infty \frac{\mu(n)}{n} (\gamma + \log \log x - \log n) + O((\log \log x)^2) \\
& =  0 + 0 - (-1) + O((\log \log x)^2) \\
& =  O((\log \log x)^2) \ (x \to \infty).
\end{align*}
The lemma follows.
\end{proof}

\begin{proposition}\label{RAprop}
One has the asymptotic expansion
\begin{align*}
\Ri(x) \sim \sum_{n = 1}^\infty \frac{\mu(n)}{n} \li (x^{1/n}) \ (x \to \infty).
\end{align*}
\end{proposition}

\begin{proof}
Let $N > 1$ be a fixed positive integer.  By the lemma, for $x > e^{N}$ one has
\begin{align*}\Ri(x) & = \sum_{n \leq \log x} \frac{\mu(n)}{n} \li (x^{1/n}) + O( (\log \log x)^2)  \\
  & = \sum_{n = 1}^{N} \frac{\mu(n)}{n} \li (x^{1/n}) + \sum_{N< n \leq \log x } \frac{\mu(n)}{n} \li (x^{1/n}) + O( (\log \log x)^2) \\
& = \sum_{n = 1}^{N} \frac{\mu(n)}{n} \li (x^{1/n}) + O\left( \li (x^{1/(N+1)})\log x\right) + O( (\log \log x)^2) \\
& = \sum_{n = 1}^{N} \frac{\mu(n)}{n} \li (x^{1/n}) + O\left( x^{1/(N+1)}\right) \\
& = \sum_{n = 1}^{N-1} \frac{\mu(n)}{n} \li (x^{1/n}) + O\left( \frac{\mu(N)}{N} \li (x^{1/N}) \right) \ (x \to \infty).
\end{align*}
The proposition follows.
\end{proof}

\begin{corollary}\label{etaconj}
One has
\begin{align*} \li(x)-\Ri(x) \sim \frac{\sqrt{x} }{\log x} \ (x \to \infty)
\end{align*}
and
\begin{align*}
\li(x)-\Ri(x) - \frac{1}{2}\li(\sqrt{x}) \sim \frac{1}{3}\li(\sqrt[3]{x}) \sim \frac{\sqrt[3]{x}}{\log x} \ (x \to \infty).
\end{align*}
\end{corollary}

By (\ref{RE}) and M\"obius inversion, one has
\begin{align*}
\li(e^x) = \sum_{n \leq x} \frac{1}{n} R(e^{x/n}), \quad x > 0
\end{align*}
and therefore
\begin{align*}
\li(x) = \sum_{n \leq \log x} \frac{1}{n} R(x^{1/n}), \quad x > 1.
\end{align*}
However, for all $x  > 0$, one has $\lim_{n \to \infty} \Ri(x^{1/n})  =  \Ri(1) =  1$, so that the sum $\sum_{n = 1}^\infty \frac{1}{n} \Ri (x^{1/n})$ diverges for all $x > 0$.   Nevertheless, one has the following.

\begin{proposition}\label{RApropa}
One has the (divergent) asymptotic expansion
\begin{align*}
\li(x) \sim \sum_{n = 1}^\infty \frac{1}{n} \Ri (x^{1/n}) \ (x \to \infty).
\end{align*}
\end{proposition}

\begin{proof}
Let $N $ be a fixed positive integer. For $x > e^{N}$ one has
\begin{align*}
\li(x) - \sum_{n = 1}^N  \frac{1}{n} \Ri (x^{1/n}) &  =  \sum_{N < n \leq \log x} \frac{1}{n} \Ri (x^{1/n})+ \sum_{n \leq \log x}  \frac{1}{n} (R(x^{1/n}) - \Ri (x^{1/n}))\\
 & = \sum_{N < n \leq \log x} \frac{1}{n} \Ri (x^{1/n}) + O\left( \sum_{n \leq \log x} \frac{(\log \log x - \log n)^2}{n} \right) \\
 & = \sum_{N < n \leq \log x} \frac{1}{n} \Ri (x^{1/n}) + O((\log \log x)^3) \\
 & \sim  \frac{1}{N+1} \Ri (x^{1/(N+1)}) \ (x \to \infty).
\end{align*}
The proposition follows.
\end{proof}


\begin{remark}
Using Riemann's approximation $\Ri(x)$ to $\pi(x)$, we can provide a plausible explanation for Legendre's approximation $L \approx 1.08366$  of the Legendre constant $L = \lim_{x \to \infty} A(x) =  1$, where 
$$A(x) = \log x-\frac{x}{\pi(x)}, \quad x > 0$$
is the unique function such that $\pi(x) = \frac{x}{\log x- A(x)}$ for  all $x > 0$.
 Figure \ref{eureka0c} compares Riemann's approximation $\Ri(x)$ with Gauss's approximation $\li(x)$, on a lin-log scale.   Notice that the graph of $x-1-\frac{e^x}{\Ri(e^x)}$ consistently traces the ``center'' of the wiggly graph of $A(e^x)-1 = x-1-\frac{e^x}{\pi(e^x)}$ and is a better approximation, at least for small $x$, than is $x-1-\frac{e^x}{\li(e^x)}$.  Figure \ref{primes15} compares the functions $x-\frac{e^x}{\Ri(e^x)}$ and $A(e^x) = x-\frac{e^x}{\pi(e^x)}$ on a smaller interval.  It is interesting to observe that the function 
$\log x-\frac{x}{\Ri(x)}$, which is Riemann's approximation to $A(x)$,  appears to attain a global maximum of approximately $1.08356$ at $x \approx 216811 \approx e^{12.2871}$, with a very small derivative nearby that appears to attain a local (and perhaps even global) minimum of only about $-3.68 \times 10^{-9}$ somewhat near the point $(475000, 1.0828)$.   These features offer a plausible explanation of how Legendre was led to his approximation $L \approx 1.08366$.   See Figure \ref{primes12} for a graph of the derivative of $\log x-\frac{x}{\Ri(x)}$ near its apparent local minimum.   
\begin{figure}[ht!]
\centering
\includegraphics[width=100mm]{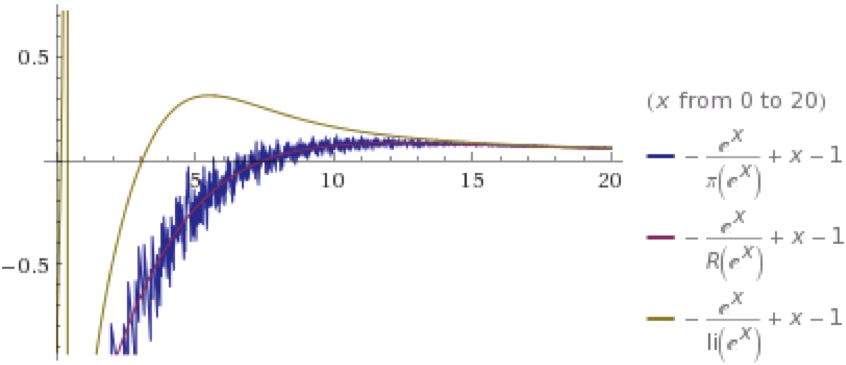}
\caption{Graph of $x-1-\frac{e^x}{\pi(e^x)}$,  $x-1-\frac{e^x}{\Ri(e^x)}$, and  $x-1-\frac{e^x}{\li(e^x)}$   \label{eureka0c}}
\end{figure}
\begin{figure}[ht!]
\centering
\includegraphics[width=100mm]{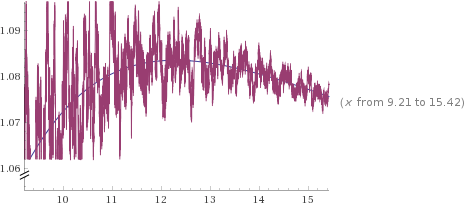}
\caption{Graph of $x-\frac{e^x}{\Ri(e^x)}$ and $x-\frac{e^x}{\pi(e^x)}$  on $[\log(10^4), \log(10^6)]$ \label{primes15}}
\end{figure}
\begin{figure}[ht!]
\centering
\includegraphics[width=110mm]{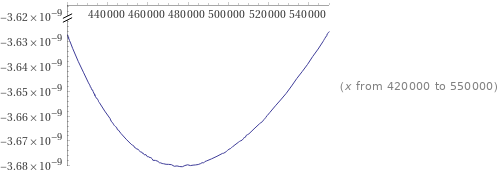}
\caption{Graph of $\frac{d}{dx}\left(\log x-\frac{x}{\Ri(x)}\right)$ \label{primes12}}
\end{figure}
\end{remark}

\subsection{Prime power counting functions}

For all $x > 0$, let $$\pi^*(x) = \sum_{k = 1}^\infty \sum_{p^k \leq x} 1$$ denote the number of prime powers (excluding $1$) less than or equal to $x$, so that
$$\pi^*(x) = \sum_{n = 1}^\infty \pi(\sqrt[n]{x}) = \sum_{n \leq \log_2 x} \pi(\sqrt[n]{x})$$
for all $x > 0$, and also let
$$\widetilde{\pi}(x) = \pi^*(x)-\pi(x) = \sum_{n = 2}^\infty \sum_{p^n \leq x} 1 =  \sum_{n = 2}^\infty \pi(\sqrt[n]{x})  =  \sum_{1 < n \leq \log_2 x} \pi(\sqrt[n]{x})$$
denote the number of composite prime powers less than or equal to $x$.   By M\"obius inversion, one has
$$\pi(x) = \sum_{n  \leq \log_2 x} \mu(n) \pi^*(\sqrt[n]{x}) =  \sum_{n = 1}^\infty \mu(n) \pi^*(\sqrt[n]{x}).$$
 One easily verifies the following analogue of Propositions \ref{RAprop2} and \ref{RAprop3}.

\begin{proposition}
One has the asymptotic expansions
$$\pi^*(x) \sim \sum_{n = 1}^\infty \pi(\sqrt[n]{x}) \ (x \to \infty)$$
and
$$\pi(x) \sim \sum_{n = 1}^\infty \mu(n)\pi^*(\sqrt[n]{x}) \ (x \to \infty)$$
\end{proposition}

In general, for any $O$ bound, one may seek explicit $O$ constants.  For example, by \cite[Lemma 3]{pan2}, one has
$$\pi(x) \leq \pi^*(x) \leq \pi(x) + \pi(\sqrt{x}) + 3\sqrt[3]{x}$$
for all $x \geq 9621$.   Thus, explicit $O$ constants can be sought for any of the terms of any of the asymptotic expansions proved in \cite{ell} and in this paper.   We do not pursue this extensive line of research here, since we are interested in pursuing asymptotic expansions rather than explicit inequalities.

\section{Asymptotic continued fraction expansions}

\subsection{Weighted prime power counting functions}

 It follows from \cite[Lemma 2.1]{ell}, the prime number theorem with error term, and our  results in Section 3 that, with respect to the asymptotic sequence $\{\frac{1}{(\log x)^n}\}$, the functions $\pi(x)$, $\Pi(x)$, $\li(x)$, and $\Ri(x)$ all have the same asymptotic continued fraction expansions, as described by \cite[Theorems 1.1 and 1.2]{ell}.  Similarly, one has the following.

\begin{theorem}
Let $n$ be a positive integer, and let $f(x)$ be any of the following functions.
\begin{enumerate}
\item $\displaystyle \Pi(x)- \sum_{k = 1}^{n-1} \frac{1}{k}\pi(\sqrt[k]{x}) = \sum_{k = n}^{\infty} \frac{1}{k}\pi(\sqrt[k]{x})$.
\item $\displaystyle \li(x)- \sum_{k = 1}^{n-1} \frac{1}{k}\Ri(\sqrt[k]{x})$.
\item $\displaystyle \mu(n)\left(\pi(x)-\sum_{k = 0}^{n-1} \frac{\mu(k)}{k}\Pi(\sqrt[k]{x})\right) = \mu(n)\left(\sum_{k = n}^{\infty} \frac{\mu(k)}{k}\Pi(\sqrt[k]{x})\right)$.
\item $\displaystyle \mu(n)\left( \Ri(x)-\sum_{k = 1}^{n-1} \frac{\mu(k)}{k}\li(\sqrt[k]{x})\right) =\mu(n)\left( \sum_{k = n}^{\infty} \frac{\mu(k)}{k}\li(\sqrt[k]{x})\right)$.
\item $\displaystyle\frac{1}{n}\pi(\sqrt[n]{x})$.
\item $\displaystyle\frac{1}{n}\Pi(\sqrt[n]{x})$.
\item $\displaystyle\frac{1}{n}\Ri(\sqrt[n]{x})$.
\item $\displaystyle \frac{1}{n}\li(\sqrt[n]{x})$.
\item $\displaystyle \frac{1}{n}\pi^*_n(x)$, where $\displaystyle \pi^*_n(x) =  \sum_{k = n}^{\infty} \pi(\sqrt[k]{x}) = \sum_{k = n}^\infty \sum_{p^k \leq x} 1$.
\item $\displaystyle \sum_{k = n}^{\infty} \frac{1}{k}\pi(\sqrt[k]{x}) = \sum_{k = n}^\infty \sum_{p^k \leq x} \frac{1}{k}$.
\end{enumerate}
One has the asymptotic continued fraction expansions
$$\displaystyle f(x) \sim \cfrac{\frac{ \sqrt[n]{x}}{\log x}}{1 \,-} \  \cfrac{\frac{n}{\log x}}{1 \,-} \  \cfrac{\frac{n}{\log x}}{1 \,-}\  \cfrac{\frac{2n}{\log x}}{1 \,-}\  \cfrac{\frac{2n}{\log x}}{1 \,-} \  \cfrac{\frac{3n}{\log x}}{1 \,-}\  \cfrac{\frac{3n}{\log x}}{1 \,-} \ \cdots \ (x \to \infty)$$
and
$$\displaystyle f(x)  \sim \cfrac{\sqrt[n]{x}}{\log x-n \,-} \  \cfrac{n^2}{\log x-3n \,-} \  \cfrac{(2n)^2}{\log x-5n \,-}\  \cfrac{(3n)^2}{\log x-7n \,-} \ \cfrac{(4n)^2}{\log x-9n \,-}\  \cdots \ (x \to \infty).$$
Consequently, the best rational approximations of the function $e^{-x/n}f(e^x)$ are precisely the approximants $w_k(x)$ of the continued fraction
$$\cfrac{1}{x-n \,-} \  \cfrac{n^2}{x-3n \,-} \  \cfrac{(2n)^2}{x-5n \,-}\  \cfrac{(3n)^2}{x-7n \,-} \  \cfrac{(4n)^2}{x-4n \,-} \ \cdots.$$
Moreover, one has
$$e^{-x/n}f(e^x) - w_k(x) \sim \frac{(n^k k!)^2}{x^{2k+1}}$$
for all $n  \geq 0$.
\end{theorem}

\begin{corollary}
Let $f(x)$ be any of the following functions.
\begin{enumerate}
\item $\displaystyle \Pi(x)-\pi(x) = \sum_{k = 2}^{\infty} \frac{1}{k}\pi(\sqrt[k]{x}) = -\sum_{k = 2}^{\infty} \frac{\mu(k)}{k}\Pi(\sqrt[k]{x})$.
\item $\displaystyle \li(x)- \Ri(x)   = - \sum_{k = 2}^{\infty} \frac{\mu(k)}{k}\li(\sqrt[k]{x})$.
\item $\displaystyle \frac{1}{2}\widetilde{\pi}(x)$, where $\displaystyle \widetilde{\pi}(x) = \sum_{k = 2}^{\infty} \pi(\sqrt[k]{x}) =  \sum_{k = 2}^\infty \sum_{p^k \leq x} 1$.
\item $\displaystyle\frac{1}{2}\pi(\sqrt{x})$.
\item $\displaystyle\frac{1}{2}\Pi(\sqrt{x})$.
\item $\displaystyle\frac{1}{2}\Ri(\sqrt{x})$.
\item $\displaystyle \frac{1}{2}\li(\sqrt{x})$.
\end{enumerate}
One has the asymptotic continued fraction expansions
\begin{align*}
f(x) \sim \cfrac{\frac{ \sqrt{x}}{\log x}}{1 \,-} \  \cfrac{\frac{2}{\log x}}{1 \,-} \  \cfrac{\frac{2}{\log x}}{1 \,-}\  \cfrac{\frac{4}{\log x}}{1 \,-}\  \cfrac{\frac{4}{\log x}}{1 \,-} \  \cfrac{\frac{6}{\log x}}{1 \,-}\  \cfrac{\frac{6}{\log x}}{1 \,-} \ \cdots \ (x \to \infty)
\end{align*}
and
\begin{align*}
f(x)   \sim \cfrac{\sqrt{x}}{\log x-2 \,-} \  \cfrac{4}{\log x-6 \,-} \  \cfrac{4 \cdot 4}{\log x-10 \,-}\  \cfrac{4 \cdot 9}{\log x-14 \,-} \ \cfrac{4 \cdot 16}{\log x-18 \,-}\  \cdots  \ (x \to \infty).
\end{align*}
Consequently, the best rational approximations of the function $e^{-x/2}f(e^x)$  are precisely the approximants of the continued fraction
$$\cfrac{1}{x-2 \,-} \  \cfrac{4}{x-6 \,-} \  \cfrac{4 \cdot 4}{x-10 \,-}\  \cfrac{4 \cdot 9}{x -14 \,-} \ \cfrac{4 \cdot 16}{x-18 \,-}\  \cdots.$$
\end{corollary}

\subsection{Sums of $s$th powers of primes}

Consider the function
$$\pi_s(x) = \sum_{p \leq x} p^s, \quad x > 0$$
for complex values of $s$ (so of course $\pi(x) = \pi_0(x)$).  The following $O$ bound is proved using the prime number theorem with error term and Abel's summation formula.

\begin{proposition}
For all $s \in \CC$ with $\operatorname{Re}(s) > -1$ and all $t > 0$, one has
$$\pi_s(x) = -E_1(-(s+1)\log x) + O \left(x^{\operatorname{Re}(s)+1}(\log x)^{-t} \right) \ (x \to \infty).$$
\end{proposition}

Consequently, \cite[Theorem 1.1, Lemma 2.1, and Corollary 3.1]{ell} yield the following.

\begin{theorem}
Let $s \in \CC$ with $\operatorname{Re}(s) > -1$.  One has the asymptotic expansion
 \begin{align*}
\pi_s(x) \sim \sum_{k = 0}^\infty \frac{k!x^{s+1}}{((s+1)\log x)^{k+1}} \ (x \to \infty)
 \end{align*}
and the asymptotic continued fraction expansions
 \begin{align*}
\pi_s(x) \sim \cfrac{\frac{x^{s+1}}{(s+1)\log x}}{1 \,-} \  \cfrac{\frac{1}{(s+1)\log x}}{1 \,-} \  \cfrac{\frac{1}{(s+1)\log x}}{1 \,-}\  \cfrac{\frac{2}{(s+1)\log x}}{1 \,-}\  \cfrac{\frac{2}{(s+1)\log x}}{1 \,-} \  \cfrac{\frac{3}{(s+1)\log x}}{1 \,-}\  \cfrac{\frac{3}{(s+1)\log x}}{1 \,-} \ \cdots \ (x \to \infty).
\end{align*}
and
\begin{align*}
\pi_s(x)   \sim \cfrac{x^{s+1}}{(s+1)\log x-1 \,-} \  \cfrac{1}{(s+1)\log x-3 \,-} \  \cfrac{4}{{(s+1)}\log x-5 \,-}\  \cfrac{9}{{(s+1)}\log x-7 \,-} \  \cdots  \ (x \to \infty).
\end{align*}
Let $w_n(x)$ for any nonnegative integer $n$ denote the $n$th approximant of the continued fraction
$$\cfrac{1}{x-1 \,-} \  \cfrac{1}{x-3 \,-} \  \cfrac{4}{x-5 \,-}\  \cfrac{9}{x -7\,-} \ \cfrac{16}{x-9 \,-}\  \cdots.$$
For all nonnegative integers $n$, one has
\begin{align*}
\frac{\pi_s(e^{x/{(s+1)}})}{e^x}-w_n(x) \sim \frac{(n!)^2}{x^{2n+1}} \ (x \to \infty).
\end{align*}
 Moreover, $w_n(x)$ is the unique Pad\'e approximant of  $\frac{\pi_s(e^{x/{(s+1)}})}{e^x}$ at $x = \infty$ of order $[n-1,n]$, and the $w_n(x)$ for all nonnegative integers $n$ are precisely the best rational approximations of the function $\frac{\pi_s(e^{x/(s+1)})}{e^x}$.
\end{theorem}

Note that, since $$\sum_{n \leq x} n^{s} = \frac{x^{s+1}}{s+1} + O(x^s) \ (x \to \infty)$$ for all $s \in \CC$ with $\operatorname{Re}(s) > -1$, the asymptotic continued fraction expansions in the theorem can be re-expressed as
 \begin{align*}
\frac{\sum_{p \leq x} p^s}{\sum_{n \leq x} n^s} \sim \cfrac{\frac{1}{\log x}}{1 \,-} \  \cfrac{\frac{1}{(s+1)\log x}}{1 \,-} \  \cfrac{\frac{1}{(s+1)\log x}}{1 \,-}\  \cfrac{\frac{2}{(s+1)\log x}}{1 \,-}\  \cfrac{\frac{2}{(s+1)\log x}}{1 \,-} \  \cfrac{\frac{3}{(s+1)\log x}}{1 \,-}\  \cfrac{\frac{3}{(s+1)\log x}}{1 \,-} \ \cdots \ (x \to \infty).
\end{align*}
and
\begin{align*}
\frac{\sum_{p \leq x} p^s}{\sum_{n \leq x} n^s} \sim \cfrac{s+1}{(s+1)\log x-1 \,-} \  \cfrac{1}{(s+1)\log x-3 \,-} \  \cfrac{4}{{(s+1)}\log x-5 \,-}\  \cfrac{9}{{(s+1)}\log x-7 \,-} \  \cdots  \ (x \to \infty).
\end{align*}

For the boundary case $s = -1$, note that
\begin{align}\label{oneover}
\sum_{p \leq x}\frac{1}{p} = M+ \log \log x + O ((\log x)^t) \ (x \to \infty)
\end{align}
for all $t \in \RR$, where
\begin{align*}
M = \lim_{x \to \infty} \left(\sum_{p \leq x} \frac{1}{p} - \log \log x\right) = 0.2614972128476427837554\ldots
\end{align*}
is the {\bf Meissel--Mertens constant}.
 In Section 2, we also noted that the uniform distribution on $[-1,0]$ has Cauchy transform $\log(1+1/z)$ with expansions (\ref{logas1}) and (\ref{logas}).  Using this, we obtain the following.

\begin{proposition}
For all real numbers $a > 1$, one has the asymptotic continued fraction expansions
$$\sum_{a^x < p \leq a^{x+1}}  \frac{1}{p} \, \sim \,  \cfrac{\frac{1}{x}}{1 \,+}\  \cfrac{\frac{1}{x}}{2 \,+}\  \cfrac{\frac{1}{x}}{3 \,+}\   \cfrac{\frac{4}{x}}{4 \,+}\  \cfrac{\frac{4}{x}}{5 \,+} \  \cfrac{\frac{9}{x}}{6 \,+}\  \cfrac{\frac{9}{x}}{7 \,+}  \ \cfrac{\frac{16}{x}}{8 \,+}  \ \cfrac{\frac{16}{x}}{9 \,+}  \ \cdots \ (x \to \infty)$$
and
$$\sum_{x < p \leq ax} \frac{1}{p} \, \sim \,  \cfrac{\frac{1}{\log_a x}}{1 \,+}\  \cfrac{\frac{1}{\log_a x}}{2 \,+}\  \cfrac{\frac{1}{\log_a x}}{3 \,+}\   \cfrac{\frac{4}{\log_a x}}{4 \,+}\  \cfrac{\frac{4}{\log_a x}}{5 \,+} \  \cfrac{\frac{9}{\log_a x}}{6 \,+}\  \cfrac{\frac{9}{\log_a x}}{7 \,+}  \ \cfrac{\frac{16}{\log_a x}}{8 \,+}  \ \cfrac{\frac{16}{\log_a x}}{9 \,+}  \ \cdots \ (x \to \infty).$$
Moreover, for all $x > 0$, the first continued fraction converges to  $\log \left( 1+\frac{1}{x}\right)$, while, for all $x > 1$, the second continued fraction converges to $\log \left(1+\frac{1}{\log_a x}\right) = \log \log(ax)-\log \log x$.
\end{proposition}

\begin{proof}
Let $t = \log a > 0$.  By (\ref{logas}), one has the asymptotic expansion
\begin{align*}
\log (z+t)-\log z = \log \left(1+\frac{t}{z}\right) \, \sim \, \cfrac{\frac{t}{z}}{1 \,+}\  \cfrac{\frac{t}{z}}{2 \,+}\  \cfrac{\frac{t}{z}}{3 \,+}\   \cfrac{\frac{4t}{z}}{4 \,+}\  \cfrac{\frac{4t}{z}}{5 \,+} \  \cfrac{\frac{9t}{z}}{6 \,+} \  \cfrac{\frac{9t}{z}}{7 \,+} \  \cfrac{\frac{16t}{z}}{8 \,+} \  \cfrac{\frac{16t}{z}}{9 \,+}  \ \cdots \ (z \to \infty),
\end{align*}
and therefore, letting $z =  \log x$, one has the asymptotic expansion
\begin{align*}
\log \log(ax)-\log \log x \,   \sim \, \cfrac{\frac{1}{\log_a x}}{1 \,+}\  \cfrac{\frac{1}{\log_a x}}{2 \,+}\  \cfrac{\frac{1}{\log_a x}}{3 \,+}\   \cfrac{\frac{4}{\log_a x}}{4 \,+}\  \cfrac{\frac{4}{\log_a x}}{5 \,+} \  \cfrac{\frac{9}{\log_a x}}{6 \,+}\  \cfrac{\frac{9}{\log_a x}}{7 \,+}  \ \cfrac{\frac{16}{\log_a x}}{8 \,+}  \ \cfrac{\frac{16}{\log_a x}}{9 \,+}  \ \cdots \ (x \to \infty).
\end{align*}
By (\ref{oneover}), one has
$$\sum_{x < p \leq ax}\frac{1}{p}   = \sum_{p \leq ax}\frac{1}{p} - \sum_{p \leq x}\frac{1}{p} = \log \log(ax) -\log \log x + o ((\log x)^t) \  (x \to \infty)$$ 
for all $t \in \RR$.  Therefore the function $\sum_{x < p \leq ax}\frac{1}{p}$ has the same asymptotic expansion as $\log \log(ax)-\log \log x$.
\end{proof}

\begin{corollary}
For all real numbers $a > b > 0$, one has the asymptotic continued fraction expansion
$$\sum_{bx < p \leq ax}\frac{1}{p} \, \sim \,  \cfrac{\frac{1}{\log_{a/b} (bx)}}{1 \,+}\  \cfrac{\frac{1}{\log_{a/b}(bx)}}{2 \,+}\  \cfrac{\frac{1}{\log_{a/b}(bx)}}{3 \,+}\   \cfrac{\frac{4}{\log_{a/b}(bx)}}{4 \,+}\  \cfrac{\frac{4}{\log_{a/b}(bx)}}{5 \,+} \  \cfrac{\frac{9}{\log_{a/b}(bx)}}{6 \,+}\  \cfrac{\frac{9}{\log_{a/b}(bx)}}{7 \,+}   \ \cdots \ (x \to \infty).$$
\end{corollary}

It is clear that the expansion (\ref{logas1}) can be rewritten in the form
$$\log \left(1+\frac{1}{z}\right) = \cfrac{\frac{1}{z}}{1 \,+}\  \cfrac{\frac{1}{z}}{2 \,+}\  \cfrac{\frac{1}{z}}{3 \,+}\   \cfrac{\frac{2}{z}}{2 \,+}\  \cfrac{\frac{2}{z}}{5 \,+} \  \cfrac{\frac{3}{z}}{2 \,+} \  \cfrac{\frac{3}{z}}{7 \,+} \  \cfrac{\frac{4}{z}}{2 \,+} \  \cfrac{\frac{4}{z}}{9 \,+} \ \cdots, \quad z \in \CC\backslash [-1,0].$$
Thus we also have the following.

\begin{corollary}
For all real numbers $a > 1$, one has the asymptotic continued fraction expansions
$$\sum_{a^x < p \leq a^{x+1}} \frac{1}{p} \, \sim \,  \cfrac{\frac{1}{x}}{1 \,+}\  \cfrac{\frac{1}{x}}{2 \,+}\  \cfrac{\frac{1}{x}}{3 \,+}\   \cfrac{\frac{2}{x}}{2 \,+}\  \cfrac{\frac{2}{x}}{5 \,+} \  \cfrac{\frac{3}{x}}{2 \,+}\  \cfrac{\frac{3}{x}}{7 \,+}  \ \cfrac{\frac{4}{x}}{2 \,+}  \ \cfrac{\frac{4}{x}}{9 \,+}  \ \cdots \ (x \to \infty)$$
and
$$\sum_{x < p \leq ax}\frac{1}{p} \, \sim \,  \cfrac{\frac{1}{\log_a x}}{1 \,+}\  \cfrac{\frac{1}{\log_a x}}{2 \,+}\  \cfrac{\frac{1}{\log_a x}}{3 \,+}\   \cfrac{\frac{2}{\log_a x}}{2 \,+}\  \cfrac{\frac{2}{\log_a x}}{5 \,+} \  \cfrac{\frac{3}{\log_a x}}{2 \,+}\  \cfrac{\frac{3}{\log_a x}}{7 \,+}  \ \cfrac{\frac{4}{\log_a x}}{2 \,+}  \ \cfrac{\frac{4}{\log_a x}}{9 \,+}  \ \cdots \ (x \to \infty).$$
\end{corollary}

It is noteworthy that the asymptotic expansion of the function $\sum_{a^x < p \leq a^{x+1}} \frac{1}{p}$ with respect to the asymptotic sequence $\left\{\frac{1}{x^n} \right\}$ does not depend on $a$.

Now, the uniform distribution on $[-1,1]$ has Cauchy transform
\begin{align}\label{uni}\log \left(\frac{z+1}{z-1}\right) =  \cfrac{2}{z \,-}\  \cfrac{1}{3z \,-}\  \cfrac{4}{5z \,-}\   \cfrac{9}{7z \,-}\  \cfrac{16}{9z  \,-} \  \cfrac{25}{11z  \,-}\  \cfrac{36}{13z  \,-} \ \cdots, \quad z \in \CC\backslash [-1,1].
\end{align}
From this we obtain the following.

\begin{proposition}
For all real numbers $a > 1$, one has the asymptotic continued fraction expansion
$$\sum_{a^{-1}x < p \leq ax} \frac{1}{p}\, \sim \,  \cfrac{2}{\log_a x \,-}\  \cfrac{1}{3\log_a x \,-}\  \cfrac{4}{5\log_a x \,-}\   \cfrac{9}{7\log_a x \,-}\  \cfrac{16}{9\log_a x  \,-} \  \cfrac{25}{11\log_a x  \,-} \ \cdots \ (x \to \infty).$$
\end{proposition}

\begin{corollary}
For all real numbers $a > 1$, one has the asymptotic continued fraction expansions
$$\sum_{a^x < p \leq a^{x+1}} \frac{1}{p}\, \sim \,  \cfrac{2}{2x+1 \,-}\  \cfrac{1}{6x+3 \,-}\  \cfrac{4}{10x +5 \,-}\   \cfrac{9}{14x +7\,-}\  \cfrac{16}{18x +9 \,-} \ \cdots \ (x \to \infty)$$
and
$$\sum_{x < p \leq ax} \frac{1}{p} \, \sim \,  \cfrac{2}{2\log_a x+1 \,-}\  \cfrac{1}{6\log_a x+3 \,-}\  \cfrac{4}{10\log_a x +5 \,-}\   \cfrac{9}{14\log_a x +7\,-} \ \cdots \ (x \to \infty).$$
\end{corollary}

\subsection{Functions related to Mertens' theorems}

Like the Meissel--Mertens constant $M$, the constant
\begin{align*}
H  & = -\sum_p \left(\frac{1}{p} +\log\left(1- \frac{1}{p} \right) \right)  \\
 & =  \sum_p \left(\frac{1}{2p^2} + \frac{1}{3p^3} + \frac{1}{4p^4} + \cdots \right) \\
 & =  \sum_{n = 2}^\infty \frac{P(n)}{n}  \\
 & = 0.3157184519\ldots,
\end{align*}
where $$P(s) = \sum_p \frac{1}{p^s}, \quad \operatorname{Re}(s) > 1$$
is the {\bf prime zeta function},  encodes information about the primes.  Since
$$\log \zeta(s)  = \sum_{n = 1}^\infty  \frac{P(ns)}{n}, \quad \operatorname{Re}(s) > 1$$ 
(which is an immediate consequence of the Euler product representation of $\zeta(s)$), one has
$$H = \lim_{x \to 1^+} \left( \log \zeta(x) - P(x) \right)  =  \lim_{x \rightarrow 1^+} \left( \log \frac{1}{x-1} - P(x) \right).$$
The following estimates are well known for $s = 1$.

\begin{proposition}\label{mertensprop}
For all $s \in \RR$ not equal to $0$ or a prime, and for all $t \in \RR$, one has the following.
\begin{enumerate}
\item  $\displaystyle -\frac{1}{s}\sum_{p \leq x}\log \left(1-\frac{s}{p}\right) = G(s) + \log \log x + o ((\log x)^t) \  (x \to \infty)$, where $$G(s) = -\lim_{x \to \infty} \left(\frac{1}{s}\sum_{p \leq x}\log \left(1-\frac{s}{p}\right) +\log \log x\right),$$ and where $G(1) = \gamma$.
\item $\displaystyle - \sum_{p \leq x} \left(\frac{1}{p} +\frac{1}{s}\log \left(1-\frac{s}{p}\right) \right) = sH(s) + o ((\log x)^t) \  (x \to \infty)$, where 
$$H(s)   = -\frac{1}{s}\sum_p \left(\frac{1}{p} +\frac{1}{s}\log\left(1- \frac{s}{p} \right) \right),$$
and where $H(1) = H$.
\item $\displaystyle  \prod_{p \leq x}\left(1-\frac{s}{p}\right)^{-1/s}  =  e^{G(s)}\log x +  o ((\log x)^t) \  (x \to \infty)$.
\item $\displaystyle  \prod_{p \leq x}\left(1-\frac{s}{p}\right)^{-1}  =  e^{sG(s)}(\log x)^s +  o ((\log x)^t) \  (x \to \infty)$.
\end{enumerate}
\end{proposition}

\begin{proof}
We prove (1), from which the other statements readily follow.  Since $s$ is not zero or a prime, the sum $\sum_{p \leq x}\log \left(1-\frac{s}{p}\right)$ is finite for all $x > 0$.  Let $N = \max(2,\lfloor |s|\rfloor + 1)$.   From the series expansion
\begin{align}\label{logest} \log \left(1-\frac{s}{t}\right) = -\sum_{k = 1}^\infty \frac{s^k}{k t^k}, \quad |t|>|s| \end{align}
it follows that
\begin{align*} \log \left(1-\frac{s}{t}\right) = -\frac{s}{t} + O \left ( \frac{1}{t^2} \right) \ (t \to \infty). \end{align*}
It follows that the function $F(u) = \log \left(1-\frac{s}{u}\right)$ satisfies the three necessary hypotheses of Landau's theorem \cite[p.\ 201--203]{land}, and therefore one has
$$\sum_{p \leq x}\log \left(1-\frac{s}{p}\right) = A(s)+ \int_{N}^x \frac{\log \left(1-\frac{s}{t}\right)}{\log t} dt + O \left((\log x)^u \right)\ (x \to \infty)$$
for all $u \in \RR$, for some constant $A(s)$ depending on $s$.   Now, since $|t| > |s|$ for all $t \geq N$, from (\ref{logest}) it follows that
\begin{align*}
\int_N^x \frac{\log \left(1-\frac{s}{t}\right)}{\log t} dt   = B(s)-s\log \log x - \sum_{k = 1}^\infty \frac{s^{k+1}}{k+1}\li (x^{-k}) 
\end{align*}
for some constant $B(s)$ depending on $s$.  But also
$$0< - \li (1/x)  < \frac{1}{x \log x}, $$
for all $x> 1$
and therefore
$$\left|\sum_{k = 1}^\infty \frac{s^{k+1}}{k+1}{\li (x^{-k})}\right| \leq \sum_{k = 1}^\infty \frac{|s|^{k+1}}{k+1}{\li (x^{-k})} <  \sum_{k = 1}^\infty \frac{|s|^{k+1}}{k(k+1)x^k \log x} = O \left(\frac{1}{x \log x} \right) \ (x \to \infty)$$
for all $x > 1$.
Thus we have
\begin{align*}  \int_N^x \frac{\log \left(1-\frac{s}{t}\right)}{\log t} dt = B(s)-s\log \log x   +  O \left(\frac{1}{x \log x}\right)  \ (x \to \infty) \end{align*}
and therefore
\begin{align*}
\sum_{p \leq x}\log \left(1-\frac{s}{p}\right)  & = A(s)+ B(s) -s\log \log x   +    O \left( \frac{1}{x \log x}\right) +  O \left((\log x)^u \right) \ (x \to \infty) \\
 &= - sG(s) -s\log \log x  + O \left((\log x)^u\right)  \ (x \to \infty),
\end{align*}
where $G(s) = -\frac{1}{s}(A(s)-B(s))$.
By Mertens'  third theorem, we know that $G(1) = \gamma$.
\end{proof}

Note that
\begin{align*}
H(s)  & = -\frac{1}{s}\sum_p \left(\frac{1}{p} +\frac{1}{s}\log\left(1- \frac{s}{p} \right) \right)  \\
 & =  \sum_p \left(\frac{1}{2p^2} + \frac{s}{3p^3} + \frac{s^2}{4p^4} + \cdots \right) \\
 & =  \sum_{n = 0}^\infty \frac{P(n+2)}{n+2}s^n,
\end{align*}
provided that the given series converges absolutely.  In fact, for any $r \geq 0$ the sequence $ \frac{P(n+1)}{n+1}r^n$ converges monotonically to $0$ if and only if $r \leq 2$, so the radius of convergence of the series $ \sum_{n = 0}^\infty \frac{P(n+2)}{n+2}s^n$ for $s \in \CC$ is $2$, and the series converges on the entire disk $|s| \leq 2$ except at $s = 2$.

\begin{corollary}
One has the following.
\begin{enumerate}
\item For all $s \in \RR$ not equal to $0$ or a prime, one has $$G(s) = M+ sH(s).$$
\item $G(0) := \lim_{s \to 0} G(s) = M$.
\item  $H(0) := \lim_{s \to 0} H(s) = \frac{1}{2} P(2)=  G'(0) = 0.2261237100205 \ldots$.
\item One has Maclaurin series expansions
$$H(s) =  \sum_{n = 0}^\infty \frac{P(n+2)}{n+2}s^n$$
and
$$G(s) =  M+\sum_{n = 1}^\infty \frac{P(n+1)}{n+1}s^n$$
valid for all  $s \in \RR$ with $|s| \leq 2$ except $s = 2$, and both series converge for all  $s \in \CC$ with $|s| \leq 2$ except $s = 2$.
\item $\gamma = G(1) = M+H$.
\item $H =  H(1) = G(1) - G(0)$.
\item $G^{(n)}(0) = \frac{n!}{n+1}P(n+1) =  nH^{(n-1)}(0)$ for all $n \geq 1$.
\item $H^{(n)}(0) = \frac{n!}{n+2}P(n+2)$ for all $n \geq 0$.
\end{enumerate}
\end{corollary}

Note that equation $\gamma = M+H$ is a well-known relationship between the constants $\gamma$, $M$, and $H$.  By the corollary, the function $G(s)$ continuously deforms the constant $M$ to the constant $\gamma = M+H$ over the interval $[0,1]$ and extends uniquely to the analytic function $M+\sum_{n = 1}^\infty \frac{P(n+1)}{n+1}s^n$ on the closed disk $|s| \leq 2$ minus $s = 2$.  An approximation of the graph of $G(s)$ on $[-2,2)$ by the first 400 terms of its Maclaurin series, is provided in Figure \ref{graphG}.

\begin{figure}[ht!]
\centering
\includegraphics[width=100mm]{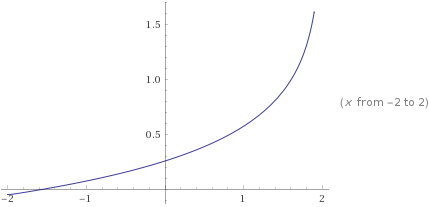}
\caption{Approximation of $G(s)$ on $[-2,2)$ by the first 400 terms of its Maclaurin series  \label{graphG}}
\end{figure}

By statements (1) and (2) of Proposition \ref{mertensprop}, for all $a > 1$ and all $s \in \RR$ not equal to $0$ or a prime, the function  $$\log \prod_{x < p \leq ax} \left(1 - \frac{s}{p} \right)^{-1/s}  =- \frac{1}{s}\sum_{x < p \leq ax} \log \left(1-\frac{s}{p}\right)$$ has the same asymptotic expansions as the function  $\sum_{x< p \leq ax} \frac{1}{p}$.  We may combine this with the results in the previous section as follows.

\begin{theorem}
Let $a > 1$.  One has the following asymptotic expansions.
\begin{enumerate}
\item $\displaystyle \ \sum_{a^x < p \leq a^{x+1}} \frac{1}{p}\, \sim \,  \sum_{n = 1}^\infty \frac{(-1)^{n-1}}{nx^{n}} \ (x \to \infty).$
\item $\displaystyle \ \sum_{a^x < p \leq a^{x+1}} \frac{1}{p}\, \sim \,  \cfrac{\frac{1}{x}}{1 \,+}\  \cfrac{\frac{1}{x}}{2 \,+}\  \cfrac{\frac{1}{x}}{3 \,+}\   \cfrac{\frac{4}{x}}{4 \,+}\  \cfrac{\frac{4}{x}}{5 \,+} \  \cfrac{\frac{9}{x}}{6 \,+}\  \cfrac{\frac{9}{x}}{7 \,+}  \ \cfrac{\frac{16}{x}}{8 \,+}  \ \cfrac{\frac{16}{x}}{9 \,+}  \ \cdots \ (x \to \infty).$
\item $\displaystyle \ \sum_{a^x < p \leq a^{x+1}} \frac{1}{p}\, \sim \,  \cfrac{\frac{1}{x}}{1 \,+}\  \cfrac{\frac{1}{x}}{2 \,+}\  \cfrac{\frac{1}{x}}{3 \,+}\   \cfrac{\frac{2}{x}}{2 \,+}\  \cfrac{\frac{2}{x}}{5 \,+} \  \cfrac{\frac{3}{x}}{2 \,+}\  \cfrac{\frac{3}{x}}{7 \,+}  \ \cfrac{\frac{4}{x}}{2 \,+}  \ \cfrac{\frac{4}{x}}{9 \,+}  \ \cdots \ (x \to \infty)$.
\item $\displaystyle \ \sum_{a^x < p \leq a^{x+1}} \frac{1}{p}\, \sim \,  \cfrac{2}{2x+1 \,-}\  \cfrac{1}{6x+3 \,-}\  \cfrac{4}{10x +5 \,-}\   \cfrac{9}{14x +7\,-}\  \cfrac{16}{18x +9 \,-} \ \cdots \ (x \to \infty).$
\item $\displaystyle \ \sum_{x < p \leq ax} \frac{1}{p}\, \sim \,  \sum_{n = 1}^\infty \frac{(-1)^{n-1}}{n(\log_a x)^{n}} \ (x \to \infty).$
\item $\displaystyle \sum_{x< p \leq ax} \frac{1}{p}\, \sim \,  \cfrac{\frac{1}{\log_a x}}{1 \,+}\  \cfrac{\frac{1}{\log_a x}}{2 \,+}\  \cfrac{\frac{1}{\log_a x}}{3 \,+}\   \cfrac{\frac{4}{\log_a x}}{4 \,+}\  \cfrac{\frac{4}{\log_a x}}{5 \,+} \  \cfrac{\frac{9}{\log_a x}}{6 \,+}\  \cfrac{\frac{9}{\log_a x}}{7 \,+}  \ \cfrac{\frac{16}{\log_a x}}{8 \,+}  \ \cfrac{\frac{16}{\log_a x}}{9 \,+}  \ \cdots \ (x \to \infty)$
\item $\displaystyle \sum_{x< p \leq ax} \frac{1}{p}\, \sim \,  \cfrac{\frac{1}{\log_a x}}{1 \,+}\  \cfrac{\frac{1}{\log_a x}}{2 \,+}\  \cfrac{\frac{1}{\log_a x}}{3 \,+}\   \cfrac{\frac{2}{\log_a x}}{2 \,+}\  \cfrac{\frac{2}{\log_a x}}{5 \,+} \  \cfrac{\frac{3}{\log_a x}}{2 \,+}\  \cfrac{\frac{3}{\log_a x}}{7 \,+}  \ \cfrac{\frac{4}{\log_a x}}{2 \,+}  \ \cfrac{\frac{4}{\log_a x}}{9 \,+}  \ \cdots \ (x \to \infty).$
\item $\displaystyle \sum_{x< p \leq ax} \frac{1}{p}\, \sim \,  \cfrac{2}{2\log_a x+1 \,-}\  \cfrac{1}{6\log_a x+3 \,-}\  \cfrac{4}{10\log_a x +5 \,-}\   \cfrac{9}{14\log_a x +7\,-}\ \cdots \ (x \to \infty).$
\end{enumerate}
Let $s \in \RR$ be nonzero and not equal to a prime.  Then the asymptotic continued fraction expansions in (1)--(4) also hold for the functions $\log(1 +\frac{1}{x})$ and $$\log \prod_{a^x < p \leq a^{x+1}} \left( 1-\frac{s}{p}\right)^{-1/s} = -\frac{1}{s}\sum_{a^x < p \leq a^{x+1}} \log \left( 1-\frac{s}{p}\right),$$
while the continued fraction expansions in (5)--(8) also hold for the functions  $\log(1 +\frac{1}{\log_a x}) = \log \log(ax)-
\log \log x$ and $$\log \prod_{x < p \leq ax} \left( 1-\frac{s}{p}\right)^{-1/s} = -\frac{1}{s}\sum_{x < p \leq ax} \log \left( 1-\frac{s}{p}\right).$$ 
\end{theorem}

We may rewrite the asymptotic expansion (\ref{uni}) as the asymptotic Jacobi continued fraction expansion
\begin{align}\label{uni2}
\log \left(\frac{z+1}{z-1}\right) =  \cfrac{2}{z \,-}\  \cfrac{\sfrac{1}{3}}{z \,-}\  \cfrac{\sfrac{4}{15}}{z \,-}\   \cfrac{\sfrac{9}{35}}{z \,-}\  \cfrac{\sfrac{16}{63}}{z  \,-} \  \cfrac{\sfrac{25}{99}}{z  \,-}\  \cfrac{\sfrac{36}{143}}{z  \,-} \cdots, \quad z \in \CC\backslash [-1,1],
\end{align}
so, substituting $z = 2x+1$,  the asymptotic expansion (4) of the theorem can be rewritten as the asymptotic Jacobi continued fraction expansion
$$\sum_{a^x < p \leq a^{x+1}} \frac{1}{p}\, \sim \,  \cfrac{1}{x+\sfrac{1}{2} \,-}\  \cfrac{\sfrac{1}{4\cdot 3}}{x+\sfrac{1}{2} \,-}\  \cfrac{\sfrac{4}{4\cdot 15}}{x+\sfrac{1}{2} \,-}\   \cfrac{\sfrac{9}{4\cdot 35}}{x+\sfrac{1}{2}\,-}\  \cfrac{\sfrac{16}{4\cdot 63}}{x+\sfrac{1}{2} \,-} \ \cdots \ (x \to \infty).$$
It is known that the denominator in the $n$th approximant of the Jacobi continued fraction in (\ref{uni2}) is the $n$th Legendre polynomial $P_n(z)$.  It follows that the denominator in the $n$th approximant of the continued fraction in the expansion of $\sum_{a^x < p \leq a^{x+1}} \frac{1}{p}$ above is the integer polynomial $\widehat{P}_n(x) = P_n(2x+1)$.  It is known that these polynomials are given explicitly by
$$\widehat{P}_n(x) = P_n(2x+1) = \sum _{k=0}^{n}{\binom {n}{k}}{\binom {n+k}{k}}x^{k}$$
for all $n$.  Now, applying \cite[Theorem 2.4]{ell}, we obtain the following.

\begin{corollary}
Let $a > 1$,  and let $s \in \RR$ be nonzero and not equal to a prime.  Let $f(x)$ denote any of the three functions $\log(1+1/x)$, $\sum_{a^x < p \leq a^{x+1}} \frac{1}{p}$, and $\log \prod_{a^x < p \leq a^{x+1}} (1 -s/p)^{-1/s}$.  Then $f(x)$ has the asymptotic Jacobi continued fraction expansion
\begin{align*}
 f(x) \, \sim  \, \cfrac{2}{2x+1 \,-}\  \cfrac{1}{6x+3 \,-}\  \cfrac{4}{10x +5 \,-}\   \cfrac{9}{14x +7\,-}\  \cfrac{16}{18x +9 \,-} \ \cdots \ (x \to \infty),
\end{align*}
The best rational appoximations of the function $f(x)$ are precisely the approximants $w_n(x)$ of the given continued fraction for $n \geq 0$, which converge to $\log (1+1/x)$ for all $x\in \CC\backslash [-1,0]$ as $n \to \infty$.   Moreover, one has
 $$f(x)- w_n(x) \sim \frac{c_n}{x^{2n+1}}  \ (x \to \infty)$$
for all $n \geq 0$,
where $c_0 = 1$ and
$$c_n = \frac{1}{2^{2n}} \prod_{k = 1}^n \frac{k^2}{4k^2-1} = \frac{1}{(2n+2){2n+1 \choose n}{2n-1 \choose n}}$$ for all $n \geq 1$.
Furthermore, one has
$$w_n(x) = \sum_{k = 1}^n \frac{c_{k-1}}{P_k(2x+1)P_{k-1}(2x+1)}$$
for all $n \geq 0$, where $P_n(x)$ denotes the $n$th Legendre polynomial and
$$P_n(2x+1) = \sum _{k=0}^{n}{\binom {n}{k}}{\binom {n+k}{k}}x^{k}$$
for all $n \geq 0$.
\end{corollary}

\begin{remark}
From \cite[{[1.14]}]{akh}, one deduces that the numerator $\widehat{R}_n(x)$ of $w_n(x) = \frac{ \widehat{R}_n(x)}{\widehat{P}_n(x)}$ is
$$\widehat{R}_n(x) =  \sum_{k = 1}^{n} a_{n,k} x^{k-1},$$
where $$a_{n,k} =  \sum_{j = k}^n \frac{(-1)^{j-k}}{j-k+1} {n \choose j}{n+j \choose j}   = {n+k \choose 2k}{2k \choose k} {}_{4}F_{3}(1,1,k-n, n+k+1; 2, k+1, k+1; 1)$$
for all $n \geq 1$ and $1 \leq k \leq n$.  
\end{remark}

\subsection{$\pi(ax)-\pi(bx)$ for $a > b > 1$}

Let $s<t$ be real numbers.  Consider the measure $\mu$ on $[s,t]$ of density $e^{-u} \, du$.  The $n$th moment of $\mu$ is
$$m_n(\mu) = \int_{s}^t u^n e^{-u} \, du = \int_{s}^\infty u^n e^{-u} \, du -\int_{t}^\infty u^n e^{-u} \, du = e^{-s} r_n(s)-e^{-t} r_n(t),$$
where $r_n(X) = \sum_{k = 1}^n \frac{n!}{k!} X^k \in \ZZ[X]$.  Moreover, one has the asymptotic expansion
$$\frac{\li(e^{x-s})-\li(e^{x-t})}{e^x} \sim \sum_{n = 0}^\infty \frac{m_n(\mu)}{x^{n+1}} \ (x \to \infty),$$
and the same expansion holds for the function $\frac{\pi(e^{x-s})-\pi(e^{x-t})}{e^x}$.
The Stieltjes transform of $\mu$ is
$${\mathcal S}_\mu(z) = -e^{-z}\left(E_1(-z+s)-E_1(-z+t)\right), \quad z \in \CC\backslash[s,t],$$
and one has
$${\mathcal S}_\mu(x) = -e^{-x}\left(E_1(-x+s)-E_1(-x+t)\right)  = \frac{\li(e^{x-s})-\li(e^{x-t})}{e^x}, \quad x \in \RR\backslash[s,t].$$
It follows that two (convergent) continued fraction expansions of ${\mathcal S}_\mu(x)$ provide  asymptotic continued fraction expansions of both $\frac{\li(e^{x-s})-\li(e^{x-t})}{e^x}$ and $\frac{\pi(e^{x-s})-\pi(e^{x-t})}{e^x}$ as $x \to \infty$.  These take the form
$$z{\mathcal S}_\mu(z) = c_0(s,t)+ \cfrac{\frac{c_1(s,t)}{z}}{1\ +} \ \cfrac{\frac{c_2(s,t)}{z}}{1\ +} \ \cfrac{\frac{c_3(s,t)}{z}}{1\ +} \  \cdots, \quad z \in \CC\backslash[s,t]$$
and
$$z{\mathcal S}_\mu(z) = a_0(s,t)+ \cfrac{b_1(s,t)}{z + a_1(s,t)\ +} \ \cfrac{b_2(s,t)}{z+a_2(s,t)\ +} \ \cfrac{b_3(s,t)}{z+a_3(s,t)\ +} \  \cdots, \quad z \in \CC\backslash[s,t],$$
where using the qd-algorithm \cite[Section 6.1]{cuyt} we compute
$$c_0(s,t) = a_0(s,t) = e^{-s}-e^{-t},$$
$$c_1(s,t) = b_1(s,t) = e^{-s}(1+s)-e^{-t}(1+t),$$
$$c_2(s,t) = a_1(s,t) =  \frac{e^{-s}(2+2s+s^2)-e^{-t}(2+2t+t^2)}{e^{-s}(1+s)-e^{-t}(1+t)}$$
$$c_3(s,t) = \frac{e^{-2s}(2+4s+s^2)+e^{-2t}(2+4t+t^2)-e^{-s-t}g(s,t)}{(e^{-s}(1+s)-e^{-t}(1+t))(e^{-s}(2+2s+s^2)-e^{-t}(2+2t+t^2))},$$
and
$$b_2(s,t) = c_2(s,t)c_3(s,t),$$
where $g(s,t)$ is the symmetric polynomial
$$g(s,t) = s^3t+st^2-2s^2t^2+s^3+t^3-s^2t-st^2-s^2-t^2+4st+4s+4t+4.$$
It follows that one has asymptotic expansions of the form
$$\frac{\pi(e^{x-s})-\pi(e^{x-t})}{e^x/x} \sim e^{-s}-e^{-t}+ \cfrac{e^{-s}(1+s)-e^{-t}(1+t)}{x\ +} \ \cfrac{c_2(s,t)}{1\ +} \ \cfrac{c_3(s,t)}{x\ +} \  \cdots  \ (x \to \infty)$$
and
$$\frac{\pi(e^{x-s})-\pi(e^{x-t})}{e^x/x} = e^{-s}-e^{-t}+ \cfrac{e^{-s}(1+s)-e^{-t}(1+t)}{x +  \frac{e^{-s}(2+2s+s^2)-e^{-t}(2+2t+t^2)}{e^{-s}(1+s)-e^{-t}(1+t)}\ +} \ \cfrac{b_2(s,t)}{x+a_2(s,t)\ +} \ \cfrac{b_3(s,t)}{x+a_3(s,t)\ +} \  \cdots \ (x \to \infty),$$
the latter of which gives explicitly the first two best rational approximations of $\frac{\pi(e^{x-s})-\pi(e^{x-t})}{e^x/x}$, where the error in the first approximation $e^{-s}-e^{-t}$  is asymptotic to $\frac{b_1(s,t)}{x}$, while the error in the second approximation is asymptotic to $\frac{b_1(s,t)b_2(s,t)}{x^3}$.

Under the obvious transformation, for $a > b > 0$ the above yields
$$\frac{\pi(ax)-\pi(bx)}{x/\log x} \sim a-b+ \cfrac{a(1-\log a)-b(1-\log b)}{\log x\ +} \ \cfrac{c_2(-\log a,-\log b)}{1\ +} \ \cfrac{c_3(-\log a,-\log b)}{\log x\ +} \  \cdots  \ (x \to \infty)$$
and
$$\frac{\pi(ax)-\pi(bx)}{x/\log x} \sim a-b+ \cfrac{a(1-\log a)-b(1-\log b)}{\log x + a_1(-\log a,-\log b)\ +} \ \cfrac{b_2(-\log a,-\log b)}{\log x+a_2(-\log a,-\log b)\ +} \  \cdots \ (x \to \infty),$$
where $c_2$, $c_3$, $a_1$, and $b_2$ are given explicitly as above. (Here, of course, $\pi(ax)-\pi(bx)$ is the number of primes $p$ such that $bx < p \leq ax$.) It also follows that the first two best approximations of $\frac{\pi(ax)-\pi(bx)}{x/\log x}$ that are rational functions of $\log x$ are $a-b$ and
$$a-b+ \cfrac{a(1-\log a)-b(1-\log b)}{\log x + \frac{a(2-2\log a+(\log a)^2)-b(2-2\log b+(\log b)^2)}{a(1-\log a)-b(1-\log b)}}.$$    For example, for $a = 2$ and $b = 1$, the second approximation is
$$\frac{\pi(2x)-\pi(x)}{x/\log x} \sim 1 - \frac{\log 4 -1}{\log x + \frac{2(\log 2 -1)^2}{\log 4 -1}} \approx   1 - \frac{0.38629436111989}{\log x + 0.48749690534099}.$$

For $a = e$ and $b = 1$ (so $s = -1$, $t = 0$, $r_n(s) = D_n$ (the number of derangements of an $n$-element set), and $r_n(t) = n!$) we computed some additional terms:
$$\frac{\pi(ex)-\pi(x)}{x/\log x} \sim e-1+ \cfrac{-1}{\log x\ +} \ \cfrac{e-2}{1\ +} \ \cfrac{-\frac{e^2-2e-2}{e-2}}{\log x\ +} \  \cfrac{-\frac{5e^2-18e+12}{(e-2)(e^2-2e-2)}}{1\ +}  \  \cfrac{\frac{(e-2)(16e^3-85e^2+104e+24)}{(e^2-2e-2)(5e^2-18e+12)}}{\log x\ +}  \  \cdots \ (x \to \infty)$$
and
$$\frac{\pi(ex)-\pi(x)}{x/\log x} \sim e-1+ \cfrac{-1}{\log x + e-2\ +} \ \cfrac{e^2-2e-2}{\log x+ \frac{8-e+2e^2-e^3}{e^2-2e-2} \ +}  \  \cfrac{-\frac{16e^3-85e^2+104e+24}{(e^2-2e-2)^2}}{\log x + a_3(-1,0) + \ } \ \cdots \ (x \to \infty).$$
Thus
$$\frac{\pi(ex)-\pi(x)}{x/\log x} -(e-1) \sim \cfrac{-1}{\log x} \ (x \to \infty),$$
$$\frac{\pi(ex)-\pi(x)}{x/\log x} -\left(e-1+ \cfrac{-1}{\log x + e-2}\right) \sim \cfrac{-1(2+2e-e^2)}{(\log x)^3} \ (x \to \infty),$$
and
$$\frac{\pi(ex)-\pi(x)}{x/\log x} -\left(e-1+ \cfrac{-1}{\log x + e-2 + \frac{e^2-2e-2}{\log x+ \frac{8-e+2e^2-e^3}{e^2-2e-2}} } \right) \sim \cfrac{ \frac{16e^3-85e^2+104e+24}{e^2-2e-2}}{(\log x)^5} \ (x \to \infty).$$
In particular, the first three best approximations of $\frac{\pi(ex)-\pi(x)}{x/\log x}$ that are rational functions of $\log x$ are $e-1$, $e-1+\frac{-1}{\log x + e-2}$, and $e-1+ \frac{-1}{\log x + e-2  + \frac{e^2-2e-2}{\log x+ \frac{8-e+2e^2-e^3}{e^2-2e-2}} }.$

\end{document}